\documentclass[a4paper,12pt]{article}

\usepackage[left=2cm,right=2cm, top=2cm,bottom=3cm,bindingoffset=0cm]{geometry}

\usepackage{verbatim}
\usepackage{amsmath}
\usepackage{amsthm}
\usepackage{amssymb}
\usepackage{delarray}
\usepackage{cite}
\usepackage{hyperref}
\usepackage{mathrsfs}

\usepackage{tikz}
\usetikzlibrary{patterns}
\usetikzlibrary{positioning,arrows,shapes,shadows}
\usetikzlibrary{matrix}
\usepackage{caption}
\DeclareCaptionLabelSeparator{dot}{. }
\newcommand{\al}{\alpha}
\newcommand{\be}{\beta}
\newcommand{\ga}{\gamma}
\newcommand{\de}{\delta}
\newcommand{\la}{\lambda}
\newcommand{\om}{\omega}

\newcommand{\eps}{\varepsilon}

\newcommand{\iy}{\infty}

\theoremstyle{plain}

\numberwithin{equation}{section}

\newtheorem{thm}{Theorem}[section]
\newtheorem{lem}[thm]{Lemma}
\newtheorem{prop}[thm]{Proposition}
\newtheorem{cor}[thm]{Corollary}

\theoremstyle{definition}

\newtheorem{example}[thm]{Example}

\newtheorem{df}[thm]{Definition}

\theoremstyle{remark}

\newtheorem{remark}[thm]{Remark}

\DeclareMathOperator*{\Res}{Res}

\DeclareMathOperator{\diag}{diag}

 \allowdisplaybreaks

\begin{document}

\begin{center}
{\Large\bf Spectral data asymptotics for the higher-order \\[0.2cm] differential  operators with distribution coefficients}
\\[0.5cm]
{\bf Natalia P. Bondarenko}
\end{center}

\vspace{0.5cm}

{\bf Abstract.}  In this paper, the asymptotics of the spectral data (eigenvalues and weight numbers) are obtained for the higher-order differential operators with distribution coefficients and separated boundary conditions.  
Additionally, we consider the case when, for the two boundary value problems, some coefficients of the differential expressions and of the boundary conditions coincide. We estimate the difference of their spectral data in this case. 

Although the asymptotic behaviour of spectral data is well-studied for differential operators with regular (integrable) coefficients, to the best of the author's knowledge, there were no results in this direction for the higher-order differential operators with distribution coefficients (generalized functions) in a general form. The technique of this paper relies on the recently obtained regularization and the Birkhoff-type solutions for differential operators with distribution coefficients. Our results have applications to the theory of inverse spectral problems as well as a separate significance.

\medskip


\medskip


\vspace{1cm}

\section{Introduction} \label{sec:intr}

Consider the differential expression
\begin{align} \nonumber
\ell_{2m+\tau}(y) := & y^{(2m + \tau)} + \sum_{k = 0}^{m-1} (-1)^{i_{2k} + k} (\sigma_{2k}^{(i_{2k})}(x) y^{(k)})^{(k)} \\ \label{defl} + & \sum_{k = 0}^{m + \tau - 2} (-1)^{i_{2k+1} + k + 1} \bigl[(\sigma_{2k+1}^{(i_{2k+1})}(x) y^{(k)})^{(k+1)} + (\sigma_{2k+1}^{(i_{2k+1})}(x) y^{(k+1)})^{(k)}\bigr],
\end{align}
where $m \in \mathbb N$, $\tau = 0, 1$, $n = 2m + \tau$; $(i_{\nu})_{\nu = 0}^{n-2}$ are integers such that $0 \le i_{2k + j} \le m - k - j$, $j = 0, 1$; $(\sigma_{\nu})_{\nu = 0}^{n-2}$ are complex-valued functions satisfying
\begin{equation} \label{condsi}
\begin{array}{ll}\sigma_{\nu} \in L_1(0, 1), \quad & \nu = \overline{0, n-2}, \\
\sigma_{2k+j} \in L_2(0,1) \quad & \text{if}\:\: n = 2m, \: i_{2k+j} = m - k - j, \: j \in \{ 0, 1 \},
\end{array}
\end{equation}
and the derivatives $\sigma_{\nu}^{(i_{\nu})}$ are understood in the sense of distributions.

The paper aims to study spectral data asymptotics for the differential equation
\begin{equation} \label{eqv}
    \ell_n(y) = \la y, \quad x \in (0, 1),
\end{equation}
subject to separated boundary conditions. The results of this paper are applied in \cite{Bond22-rec} to the inverse problem theory and also have a separate significance.

If the functions $\sigma_{\nu}(x)$ are sufficiently smooth, then the differential expression \eqref{defl} can be represented in the form
\begin{equation} \label{horeg}
y^{(n)} + \sum_{k = 0}^{n-2} q_k(x) y^{(k)},
\end{equation}
where $( q_k )_{k = 0}^{n-2}$ are some integrable functions. However, for differential operators with distribution coefficients, it is more convenient to consider the divergent form \eqref{defl} following \cite{MS16, MS19, Bond22}.

For regular differential operators \eqref{horeg}, the standard approach to obtaining eigenvalue asymptotics is described in the classical monograph by Naimark~\cite{Nai68}. In recent years, eigenvalue asymptotics of higher-order differential operators with non-smooth coefficients attract considerable attention (see, e.g., \cite{Akh11, BK12, BK21, Pol22}). 

For differential operators with distribution coefficients, the asymptotic behavior of the eigenvalues has been studied much less. In \cite{Sav01, HM04}, asymptotic formulas have been obtained for the eigenvalues of 
the Sturm-Liouville operators with potentials of class $W_2^{-1}(0,1)$ (i.e. $n = 2$, $i_0 = 1$ in \eqref{defl}). In \cite{MM04, MM12}, the eigenvalue asymptotics were studied for the even-order operator $\frac{d^{2m}}{dx^{2m}}$ perturbed by distribution potential. For the higher-order differential operators with distribution coefficients of the general form \eqref{defl}, to the best of the author's knowledge, the asymptotic behaviour of the eigenvalues has not been investigated before.

Our treatment of the differential expression \eqref{defl} relies on the regularization methods of \cite{MS16, MS19, Vlad17}. Mirzoev and Shkalikov have developed the regularization approach to the differential expression \eqref{defl} with $i_{2k+j} = m - k - j$, $j \in \{ 0, 1 \}$ for an even order $n = 2m$ in \cite{MS16} and for an odd order $n = 2m+1$ in \cite{MS19}. Vladimirov \cite{Vlad17} has obtained an alternative construction, which can be used a wider class of differential operators than the results of \cite{MS16, MS19}. In particular, the approach of \cite{Vlad17} has been applied to the differential expression of form \eqref{defl} in \cite{Bond22}. It is worth mentioning that, in \cite{MS16, MS19, Vlad17}, the coefficients at $y^{(n)}$ and $y^{(n-1)}$ in the differential expression can be arbitrary functions of some classes. In this paper, we confine ourselves to the coefficients $1$ and $0$ at $y^{(n)}$ and $y^{(n-1)}$, respectively, because this case is natural for the inverse problem theory (see \cite{Bond21, Bond22, Bond22-rec}).

Let is briefly describe the regularization of the differential expression \eqref{defl}.
The matrix function $F(x) = [f_{k,j}(x)]_{k, j = 1}^n$ associated with $\ell_n(y)$ is constructed: $F = F_{i_0, i_1, \ldots, i_{n-2}}(\sigma_0, \sigma_1, \ldots, \sigma_{n-2})$. The certain formulas for the associated matrix entries $f_{k,j}(x)$ are provided in Section~\ref{sec:reg}. By using the quasi-derivatives
\begin{equation} \label{quasi}
y^{[0]} := y, \quad y^{[k]} := (y^{[k-1]})' - \sum_{j = 1}^k f_{k,j} y^{[j-1]}, \quad k = \overline{1,n},
\end{equation}
equation \eqref{eqv} is reduced to the equivalent system
\begin{equation} \label{sys}
{\vec y}\,' = (F(x) + \Lambda) \vec y, \quad x \in (0, 1),
\end{equation}
where $\vec y(x) = \mbox{col} ( y^{[0]}(x), y^{[1]}(x), \ldots, y^{[n-1]}(x))$, $\Lambda$ is the $(n \times n)$ matrix
 whose entry at the position $(n,1)$ equals $\la$ and all the other entries equal $0$.

Define the boundary conditions
\begin{equation} \label{bc}
\left.
\begin{array}{l}
    U_s(y) := y^{[p_s]}(0) + \sum\limits_{j = 1}^{p_s} u_{s,j} y^{[j-1]}(0) = 0, \quad s = \overline{1, r}, \\
    U_s(y) := y^{[p_s]}(1) + \sum\limits_{j = 1}^{p_s} u_{s,j} y^{[j-1]}(1) = 0, \quad s = \overline{r+1, n},
\end{array} \right\}
\end{equation}
where $r \in \{ 1, \ldots, n-1 \}$ is fixed, $p_s \in \{ 0, \ldots, n-1 \}$ for $s = \overline{1, n}$, $p_s \ne p_k$ for $1 \le s < k \le r$ and for $r + 1 \le s < k \le n$. It can be easily shown that the spectrum of the boundary value problem~\eqref{eqv},\eqref{bc} is a countable set of eigenvalues (see \cite{Bond21}). In \cite{Bond21}, the spectra of several problems of form \eqref{eqv},\eqref{bc} have been used as the spectral data of the inverse problem. 

The first result of this paper is the following theorem, which describes the asymptotic behaviour of the eigenvalues.

\begin{thm} \label{thm:asympt}
The eigenvalues $\{ \la_l \}_{l \ge 1}$ of the  boundary value problem \eqref{eqv},\eqref{bc} satisfy the relation
\begin{equation} \label{asymptla}
\la_l = (-1)^{n-r} \left( \frac{\pi}{\sin\tfrac{\pi r}{n}} (l + \chi + \eps_l) \right)^n, \quad
l \in \mathbb N, \quad  \{ \eps_l \} \in l_2,
\end{equation}
where the constant $\chi$ depend only on $n$, $r$, $( p_s )_{s = 1}^n$ and do not depend on $(\sigma_{\nu})_{\nu = 0}^{n-2}$ and $u_{s,j}$, $s = \overline{1, n}$, $j = \overline{1, p_s}$.
\end{thm}

In the inverse problem theory \cite{Bond22-rec}, it is convenient to recover the coefficients $\sigma_{n-2}$, $\sigma_{n-3}$, \ldots, $\sigma_1$, $\sigma_0$ one-by-one. Therefore, the question arises:

\smallskip

\textit{If $(\sigma_{\nu})_{\nu = \nu_0}^{n-2}$ are known, then what can be said about the eigenvalue asymptotics?}

\smallskip

We give a rigorous answer to this question in Theorem~\ref{thm:dif}.
Denote by $\mathcal L$ the boundary value problem \eqref{eqv},\eqref{bc} and by $\tilde {\mathcal L}$ the boundary value problem of the same form but with the coefficients $(\sigma_{\nu})$ and $(u_{s,j})$ replaced by $(\tilde \sigma_{\nu})$ and $(\tilde u_{s,j})$, respectively.
The numbers $n$, $r$, $(i_{\nu})$, and $(p_s)$ for the problems $\mathcal L$ and $\tilde {\mathcal L}$ are the same.
Throughout the paper, if a symbol $\gamma$ denotes an object related to the problem without tilde, then $\tilde \ga$ denotes the similar object related to the problem with tilde, and $\hat \ga = \ga - \tilde \ga$.
Consider the values $\rho_l$ and $\tilde \rho_l$ from the asymptotics \eqref{asymptla} for the problems $\mathcal L$ and $\tilde{\mathcal L}$, respectively.

\begin{thm} \label{thm:dif}
Suppose that $\sigma_{\nu}(x) = \tilde \sigma_{\nu}(x)$ for a.e. $x \in (0, 1)$, $\nu = \overline{\nu_0, n-2}$. Denote
\begin{gather} \label{defd1}
d := n - 1 - \max_{\nu = \overline{0, \nu_0-1}} (\nu + i_{\nu}), \\ \nonumber
 \mathcal N_d := \{ \nu = \overline{0, \nu_0-1} \colon d = n - 1 - (\nu + i_{\nu}) \}, \quad \mathcal N_d^0 := \{ \nu \in \mathcal N_d \colon i_{\nu} = 0 \},
\end{gather}
and assume that $u_{s, p_s-j} = \tilde u_{s, p_s - j}$ for $j = \overline{0, d-2}$, $s = \overline{1, n}$.
Then
$$
\rho_l - \tilde \rho_l = l^{-d} (\hat c + \delta_l), \quad \delta_l = o(1), \quad l \to \infty,
$$
where the constant $\hat c$ depends on the numbers 
\begin{equation} \label{valc}
\int_0^1 \hat\sigma_{\nu}(x) \, dx, \quad \nu \in \mathcal N_d^0, \quad \text{and} \quad \hat u_{s, p_s - d + 1}, \quad s = \overline{1, n}.
\end{equation}

In particular, $\hat c = 0$ if all the numbers \eqref{valc} equal zero. If $\hat \sigma_{\nu} \in L_2(0,1)$, $\nu \in \mathcal N_d$, then $\{ \delta_l \} \in l_2$.
\end{thm}

Theorem~\ref{thm:dif} helps to improve the asymptotics of Theorem~\ref{thm:asympt} for odd $n$ and so leads to the following result.

\begin{cor} \label{cor:odd}
For $n = 2m+1$, the remainder $\eps_l$ in \eqref{asymptla} has the form
\begin{equation} \label{asymptodd}
\eps_l = \frac{\chi_1}{l} + \frac{\eps_{l,1}}{l}, \quad \eps_{l,1} = o(1), \quad l \to \infty,
\end{equation}
and the constant $\chi_1$ depends on $\int\limits_0^1 \sigma_{n-2}(x) \, dx$ and $u_{s, p_s}$, $s = \overline{1,n}$. If $\sigma_{n-2} \in L_2(0,1)$ and $(\sigma_{n-3} \in L_2(0,1), i_{n-3} = 1 \:\: \text{or} \:\: i_{n-3} = 0)$, then $\{ \eps_{l,1} \} \in l_2$.
\end{cor}

In order to prove Theorem~\ref{thm:asympt}, we use the approach of Naimark~\cite{Nai68} and the Birkhoff-type solutions constructed by Savchuk and Shkalikov~\cite{SS20}.
The proof of Theorem~\ref{thm:dif} relies on the special structure of the matrix function $F(x)$ associated with the differential expression~$\ell_n(y)$. We develop the technique of~\cite{SS20} to study the difference of the corresponding Birkhoff-type solutions for the problem $\mathcal L$ and $\tilde{\mathcal L}$. Further, we follow the proof strategy of Theorem~\ref{thm:asympt} to analyze the difference of the eigenvalues.

In addition, we obtain the analogs of Theorems~\ref{thm:asympt} and~\ref{thm:dif} for the weight numbers defined in Section~\ref{sec:weight}. The weight numbers together with the eigenvalues are used as spectral data for recovering higher-order differential operators with distribution coefficients in \cite{Bond22-rec}.

The paper is organized as follows. In Section~\ref{sec:reg}, equation~\eqref{eqv} is transformed to the first-order system~\eqref{sys} and then \eqref{sys} is reduced to a more convenient form for studying solution asymptotics. Section~\ref{sec:birk} is devoted to the Birkhoff-type solutions of equation \eqref{eqv} with the certain asymptotic behavior for large values of the spectral parameter. We formulate the necessary propositions from \cite{SS20} and study the difference of the Birkhoff-type solutions for the problems $\mathcal L$ and $\tilde{\mathcal L}$ satisfying the conditions of Theorem~\ref{thm:dif}. In Section~\ref{sec:eig}, the proofs of Theorems~\ref{thm:asympt},~\ref{thm:dif}, and Corollary~\ref{cor:odd} are provided. In Section~\ref{sec:examp}, the main results on the eigenvalue asymptotics are illustrated by several examples. Section~\ref{sec:weight} contains the definition of the weight numbers and the analogs of Theorems~\ref{thm:asympt},~\ref{thm:dif} for the weight numbers supplied by the proofs.

Throughout the paper, we use the following \textbf{notations}.

\begin{itemize}
    \item The same symbol $C$ denotes various constants independent of $x$, $\rho$, etc.
    \item $I$ denotes the $(n \times n)$ unit matrix.
    \item $\de_{jk}$ is the Kronecker delta.
    \item We use the following vector and matrix norms:
    $$
        \| a \| = \max_i |a_i|, \quad a = [a_i]_{i = 1}^n, \qquad
        \| A \| = \max_{i,j} |a_{ij}|, \quad A = [a_{ij}]_{i,j = 1}^n.
    $$
    \item $\diag\{ d_1, d_2, \ldots, d_n \}$ is the diagonal matrix with the entries $(d_k)_{k=1}^n$ on the main diagonal.
    \item For a matrix $A = [a_{kj}]_{k,j=1}^n$, we denote by $\diag(A)$ the diagonal matrix $\diag\{ a_{11}, a_{22}, \ldots, a_{nn} \}$.
    \item We use the same notation $L_{\mu}(0,1)$, $\mu \in [1, \infty]$, for the space of scalar functions, 
    for the space of vector functions 
    $$
        Y = [y_j]_{j = 1}^n, \quad y_j \in L_{\mu}(0,1), \quad \| Y \|_{L_{\mu}} = \max_j \| y_j \|_{L_{\mu}},
    $$
    and for the space of matrix functions
    $$
        A = [a_{kj}]_{k,j = 1}^n, \quad a_{kj} \in L_{\mu}(0,1), \quad \| A \|_{L_{\mu}} = \max_{k,j} \| a_{kj} \|_{L_{\mu}}.
    $$
\item The notation $\{ \varkappa_l \}$ is used for various sequences of $l_2$.
    \item $\la = \rho^n$, $\dot f(\rho) = \tfrac{d}{d\rho} f(\rho)$.
\end{itemize}

\section{Reduction to first-order systems} \label{sec:reg}

In this section, equation~\eqref{eqv} is reduced to the system~\eqref{sys} and then to the form \eqref{sysw}, which is more convenient for studying the asymptotics of solutions. This section is based on the results of \cite{MS16, Vlad17, SS20, Bond22}.

The associated matrix $F(x)$ defined by the coefficients $(\sigma_{\nu})_{\nu = 0}^{n-2}$ of the differential expression \eqref{defl} as follows.

\begin{df} \label{def:F}
Define the matrix $Q(x) = [q_{\xi,j}(x)]_{\xi,j = 0}^m$ by the following formulas:
\begin{gather*} 
Q(x) := \sum_{\nu = 0}^{n-2} \sigma_{\nu}(x) \chi_{\nu, i_{\nu}}, \quad \chi_{\nu, i} = [\chi_{\nu,i;\xi,j}]_{\xi,j = 0}^m, \\ \label{defchi}
\begin{array}{rl}
\chi_{2k, i; s + k, i - s + k} = & C_i^s, \quad s = \overline{0, i}, \\
\chi_{2k+1, i; s+k, i + 1- s + k} = & C_{i + 1}^s - 2 C_i^{s-1}, \quad s = \overline{0, i+1}, 
\end{array}
\end{gather*}
and all the other entries $\chi_{\nu, i; \xi,j}$ equal zero. Here and below, $C_i^s = \dfrac{i!}{s!(i-s)!}$ are the binomial coefficients, $C_i^{-1} := 0$. Then, by using the elements $q_{\xi,j}$, define the elements of the matrix function $F(x) = [f_{k,j}]_{k,j = 1}^n$ as follows:
\begin{align*}
    & n = 2m \colon \quad \begin{cases}
                            f_{m,j} := (-1)^{m+1} q_{j-1,m}, \quad j = \overline{1, m}, \\
                            f_{k,m+1} := (-1)^{k+1} q_{m,2m-k}, \quad k = \overline{m+1, 2m}, \\
                            f_{k,j} := (-1)^{k+1} q_{j-1,2m-k} + (-1)^{m+k} q_{j-1,m} q_{m,2m-k}, \quad k = \overline{m+1,2m}, \, j = \overline{1,m},
                        \end{cases} \\ 
    & n = 2m+1 \colon \quad f_{k,j} := (-1)^k q_{j-1, 2m+1-k},\, k = \overline{m+1, 2m+1}, \, j = \overline{1, m+1}.        
\end{align*}
The other elements are defined as $f_{k,j} = \de_{k+1, j}$. 
\end{df}

Definition~\ref{def:F} together with the condition \eqref{condsi} imply
\begin{equation} \label{classf}
f_{k,j} \in L_1(0,1), \quad f_{k,k} \in L_2(0,1), \quad 1 \le j \le k \le n, \quad \mbox{trace}(F(x)) = 0.
\end{equation}

Suppose that $y \in W_{2,loc}^m(0, 1)$ if for some indices $\nu \in \{ 0, \ldots, n-2 \}$ the condition \eqref{condsi} implies $\sigma_\nu \in L_2(0,1)$ and $y \in W_{1,loc}^m(0,1)$ otherwise. Then $\ell_n(y)$ is correctly defined in the space $\mathfrak D'$ of the linear functionals (generalized functions) on $\mathfrak D = C^{\iy}_0(0,1)$ (see \cite[Lemma~2.1]{Bond22}). If $y \in \mathcal D_F$,
$$
\mathcal D_F := \{ y \colon y^{[k]} \in AC_{loc}(0,1), \: k = \overline{0,n-1} \},
$$
then $\ell_n(y)$ is a regular function and $\ell_n(y) = y^{[n]}$ (see \cite[Theorem~2.2]{Bond22}). Then, instead of equation \eqref{eqv}, one can consider the system \eqref{sys}. Indeed, in view of \eqref{quasi}, the first $(n-1)$ rows of \eqref{sys} coincide with the definition of the quasi-derivatives $y^{[k]}$, $k = \overline{1, n-1}$, and the last row is $y^{[n]} = \la y$. Below, we say that $y$ is a \textit{solution} of equation \eqref{eqv} if $y \in \mathcal D_F$ and the vector function $\vec y(x) = \mbox{col} ( y^{[0]}(x), y^{[1]}(x), \ldots, y^{[n-1]}(x))$ satisfies \eqref{sys}.

Let $\la = \rho^n$.
The change of variables $\vec y(x) = \diag \{ 1, \rho, \ldots, \rho^{n-1} \} u(x)$ transforms the system \eqref{sys} into
\begin{equation} \label{sysu}
u'(x) = F(x, \rho) u, \quad x \in (0, 1),
\end{equation}
where
\begin{gather*}
F(x, \rho) = \rho F_{-1} + F_0(x) + \sum_{k = 1}^{n-1} \rho^{-k} F_k(x), \\
F_{-1} = \begin{bmatrix}
        0 & 1 & 0 & \dots & 0 & 0 \\
        0 & 0 & 1 & \dots & 0 & 0 \\
        \hdotsfor{6} \\
        0 & 0 & 0 & \dots & 1 & 0 \\
        0 & 0 & 0 & \dots & 0 & 1 \\
        1 & 0 & 0 & \dots & 0 & 0 
     \end{bmatrix},
\end{gather*}
and the matrix functions $F_k(x)$ are formed by the corresponding lower diagonals of $F(x)$. The relations \eqref{classf} imply $F_0 \in L_2(0, 1)$, $F_k \in L_1(0, 1)$, $k = \overline{1,n-1}$.

Following the standard ideas described in the book of Naimark \cite{Nai68}, we consider the partition of the $\rho$-plane into the sectors
\begin{equation} \label{defGa}
\Gamma_{\kappa} = \left\{ \rho \colon \frac{\pi(\kappa-1)}{n} < \arg \rho < \frac{\pi \kappa}{n} \right\}, \quad \kappa = \overline{1, 2n}.
\end{equation}
Fix a sector $\Gamma_{\kappa}$. Denote by $\{ \omega_k \}_{k = 1}^n$ the roots  of the equation $\omega^n = 1$ numbered so that
\begin{equation} \label{order}
\mbox{Re} \, (\rho \om_1) < \mbox{Re} \, (\rho \om_2) < \dots < \mbox{Re} \, (\rho \om_n), \quad \rho \in \Gamma_{\kappa}.
\end{equation}
We also define the extended sector (see Fig.~\ref{img:sectors}):
\begin{equation} \label{defsec}
\Gamma_{\kappa, h} := \left\{ \rho \in \mathbb C \colon \rho + h \exp\bigl( \tfrac{i \pi (\kappa - 1/2)}{n}\bigr)  \in \Gamma_{\kappa}\right\}, \quad h > 0.
\end{equation}

\begin{figure}[h!]
\centering
\begin{tikzpicture}
\draw[dotted] (-1, 0) edge (4, 0);
\draw[dotted] (0, -1) edge (0, 4);
\draw (0, 0) edge (2, 4);
\draw (0, 0) edge (4, 2);
\draw (2, 2) node{$\Gamma_{\kappa}$};
\draw (-1.2, -1.2) edge (1.8, 4.8);
\draw (-1.2, -1.2) edge (4.8, 1.8);
\draw (-0.2, -0.3) node{$\Gamma_{\kappa,h}$};
\end{tikzpicture}
\caption{Sectors}
\label{img:sectors}
\end{figure}

Put $B := \diag \{ \om_1, \om_2, \ldots, \om_n \}$, $\Omega := [\om_k^{j-1}]_{j,k = 1}^n$. Obviously, $\Omega^{-1} F_{-1} \Omega = B$. By changing the variables $w(x) := \Omega^{-1} u(x)$, we reduce the system \eqref{sysu} to the form
\begin{equation} \label{sysw}
w' = \rho B w + A(x, \rho) w, \quad x \in (0, 1),
\end{equation}
where
\begin{gather} \label{defA}
A(x, \rho) = A_0(x) + \sum_{k =1}^{n-1} \rho^{-k} A_k(x), \\ \nonumber
A_k(x) = \Omega^{-1} F_k(x) \Omega,\quad k = \overline{0,n-1}.
\end{gather}

Thus, instead of equation \eqref{eqv}, one can consider the system \eqref{sysw}. It follows from \eqref{classf} and \eqref{defA} that $A_0 \in L_2(0,1)$, $A_k \in L_1(0,1)$, $k = \overline{1,n-1}$, $\diag(A_0(x)) \equiv 0$.

\section{Birkhoff-type solutions} \label{sec:birk}

In this section, we study the Birkhoff-type solutions with certain asymptotic behavior as $|\rho| \to \infty$ for the system~\eqref{sysw} and for equation~\eqref{eqv}. 
The Birkhoff-type fundamental systems of solutions (FSS) for the first-order systems which generalize \eqref{sysw} have been constructed in \cite{Rykh99, SS20}. In this paper, we use the approach of Savchuk and Shkalikov \cite{SS20}. First, we provide the necessary notations and results of \cite{SS20} specified for the system \eqref{sysw}. Second, we consider the problems $\mathcal L$ and $\tilde{\mathcal L}$ satisfying the conditions of Theorem~\ref{thm:dif} and investigate the difference of the corresponding Birkhoff-type solutions. The main results for the latter case are formulated in Theorem~\ref{thm:Birk} and Corollary~\ref{cor:y}.

Suppose that $A(x, \rho)$ is an arbitrary matrix function of form \eqref{defA}, where $A_k \in L_1(0, 1)$, $k = \overline{0,n-1}$, $\diag(A_0(x)) \equiv 0$, $\Gamma_{\kappa, h}$ is a fixed sector of form \eqref{defsec},
$\{ \om_j \}_{j = 1}^n$ are the roots of the equation $\om^n = 1$ numbered in the order \eqref{order}.

Denote the elements of the matrices $A_k(x)$ and $A(x, \rho)$ by $a_{k,jl}$ and $v_{jl}(x, \rho) = a_{0,jl}(x) + r_{jl}(x, \rho)$, $j,l = \overline{1, n}$, respectively. Put
\begin{align} \label{defvjkl}
    v_{jkl}(s, x, \rho) & := (\pm)_{jk} (\pm)_{lk} \int a_{0,jl}(t) \exp(\rho[(\om_l - \om_k)(t-s) + (\om_j - \om_k)(x-t)])\, dt, \\ \label{defrhojkl}
    \varrho_{jkl}(s, x, \rho) & := (\pm)_{jk} (\pm)_{lk} \int r_{jl}(t, \rho) \exp(\rho[(\om_l - \om_k)(t-s) + (\om_j - \om_k)(x-t)])\, dt,
\end{align}
where $(\pm)_{jk} = \left\{\begin{array}{ll} -1, \quad & j < k \\ 1, \quad & j \ge k \end{array}\right.$, the integration is taken over the intervals
$$
\left\{ \begin{array}{ll}
(x, s), \quad & \text{if} \:\: j,l < k, \\
(\max\{ x, s \}, 1), \quad & \text{if} \:\: j < k \le l, \\
(0, \min \{ x, s \}), \quad & \text{if} \:\: l < k \le j, \\
(s, x), \quad & \text{if} \:\: k \le j,l, \\
\end{array} \right.
$$
and the integrals are assumed to be zero if the upper limit is less than the lower one. In view of \eqref{order},
the exponents in \eqref{defvjkl} and \eqref{defrhojkl} are bounded:
$$
|\exp(\rho[(\om_l - \om_k)(t-s) + (\om_j - \om_k)(x-t)])| \le C, \quad \rho \in \overline{\Gamma}_{\kappa, h}.
$$

Introduce the region $\mathcal G := \{ \rho \in \Gamma_{\kappa,h} \colon |\rho| \ge \rho^* \}$ for some $\rho^* > 0$.
In view of \eqref{defA}, we have 
$$
\| r_{jl}(., \rho) \|_{L_1} \le C |\rho|^{-1}, \quad \rho \in \overline{\mathcal G},
$$
and so
\begin{equation} \label{estrho}
\max_{j,k,l,s,x} |\varrho(s, x, \rho)| \le C |\rho|^{-1}.
\end{equation}

Denote
\begin{equation} \label{defUps1}
\Upsilon(\rho) := \max_{j,k,l,s,x} |v_{jkl}(s, x, \rho)|.
\end{equation}

\begin{prop}[\cite{SS20}] \label{prop:Birk}
For any fixed sector $\Gamma_{\kappa,h}$ and some $\rho^* > 0$, the system \eqref{sysw} has a fundamental solution matrix $w(x, \rho)$ of form
\begin{equation} \label{asymptw}
    w(x, \rho) = (I + \mathcal E(x, \rho)) \exp(\rho B x), 
\end{equation}
where $\mathcal E(x, \rho)$ is continuous for $x \in [0, 1]$, $\rho \in \overline{\mathcal G}$,
analytic in $\rho$ for each fixed $x \in [0, 1]$, $\rho \in \mathcal G$, and
\begin{equation} \label{estE}
\max_x \| \mathcal E(x, \rho) \| \le C(\Upsilon(\rho) + |\rho|^{-1}), \quad \rho \in \overline{\mathcal G}.
\end{equation}
\end{prop}

\begin{prop}[\cite{SS20}] \label{prop:Ups}
$\Upsilon(\rho) \to 0$ as $|\rho| \to \iy$, $\rho \in \Gamma_{\kappa,h}$.
\end{prop}

We call a sequence $\{ \rho_k \}_{k = 1}^{\iy}$ \textit{non-condensing} if 
$$
\beta := \sum(N(t + 1) - N(t)) < \iy, \quad N(t) := \#\{ k \in \mathbb N \colon |\rho_k| \le t \}.
$$

\begin{prop}[\cite{Sav-dis}] \label{prop:nonc}
Suppose that $A_0 \in L_{\mu}(0, 1)$, $\mu \in (1,2]$, and $\{ \rho_k \}_{k = 1}^{\iy}$ is a non-condensing sequence in $\mathcal G$. Then the sequence $\{ \Upsilon(\rho_k) \}_{k = 1}^{\iy}$ belongs to $l_{\mu'}$, $\frac{1}{\mu} + \frac{1}{\mu'} = 1$, and
$$
\| \{ \Upsilon(\rho_k) \} \|_{l_{\mu'}} \le C \| A_0 \|_{L_{\mu}},
$$
where the constant $C$ depends only on $h$, $\rho^*$, and $\beta$.
\end{prop}

Along with the system \eqref{sysw}, consider the system
\begin{gather} \label{syswt}
    \tilde w' = \rho B \tilde w + \tilde A(x, \rho) \tilde w, \quad x \in (0, 1), \\ \nonumber
    \tilde A(x, \rho) = \tilde A_0(x) + \sum_{k = 1}^{n-1} \rho^{-k} \tilde A_k(x).
\end{gather}
Suppose that $\tilde A_k \in L_1(0, 1)$, $k = \overline{0,n-1}$, and $\diag(\tilde A_0(x)) \equiv 0$.

Consider the difference $\hat w(x, \rho) = w(x, \rho) - \tilde w(x, \rho)$ of the fundamental solutions defined by Proposition~\ref{prop:Birk} for the systems \eqref{sysw} and \eqref{syswt}.

\begin{thm} \label{thm:Birk}
Suppose that
\begin{equation} \label{eqA}
A_k(x) = \tilde A_k(x) \:\: \text{a.e. on} \:\: (0, 1), \quad k = \overline{0, d-1},
\end{equation}
for a fixed $d \in \{ 1, \ldots, n-1 \}$. Then
\begin{gather} \nonumber
\hat w(x, \rho) = \hat {\mathcal E}(x, \rho) \exp(\rho B x), \\ \label{estEh}
\max_x \left \| \rho^d \hat {\mathcal E}(x, \rho) - \int_0^x \diag(\hat A_d(t)) \, dt \right\| \le C (\Upsilon(\rho) + \Upsilon_d(\rho) + |\rho|^{-1}), \quad \rho \in \overline{\mathcal G},
\end{gather}
where
\begin{gather} \label{defUps}
\Upsilon(\rho) := \max_{j,k,l,s,x} \{ |v_{jkl}(s,x,\rho)|, |\tilde v_{jkl}(s,x,\rho)|\}, \\ \label{defUpsd}
\Upsilon_d(\rho) := \max_{j\ne k; x} |\al_{jk}(x, \rho)|, \quad
\al_{jk}(x, \rho) := \int_{b_{jk}}^x \hat a_{d,jk}(t) \exp(\rho(\om_j - \om_k)(x - t)) \, dt, \\ \label{defb}
b_{jk} := \begin{cases}
            0, \quad j \ge k, \\
            1, \quad j < k.
         \end{cases}
\end{gather}
\end{thm}

Clearly, Propositions~\ref{prop:Ups} and~\ref{prop:nonc} are valid for $\Upsilon(\rho)$ defined by \eqref{defUps}. Proposition~\ref{prop:Ups} can be similarly proved for $\Upsilon_d(\rho)$. Proposition~\ref{prop:nonc} is valid for $\Upsilon_d(\rho)$ if $A_0$ is replaced with $A_d$.

\begin{proof}[Proof of Theorem~\ref{thm:Birk}]
In this proof, we apply the technique of \cite{SS20}.
By changing the variables $w(x, \rho) = z(x, \rho) \exp(\rho B x)$, $z(x, \rho) = [z_{jk}(x, \rho)]_{j,k = 1}^n$, we reduce the system \eqref{sysw} to the form
$$
z' = \rho (B z - z B) + A(x, \rho) z, \quad x \in (0, 1).
$$
By integrating the latter system with the initial conditions
$$
z_{jk}(1, \rho) = 0, \quad j < k, \qquad
z_{jk}(0, \rho) = \de_{jk}, \quad j \ge k,
$$
we obtain the integral equations
\begin{equation} \label{intz}
    z_{jk}(x, \rho) - \de_{jk} = \sum_{l = 1}^n \int_{b_{jk}}^x v_{jl}(t, \rho) \exp(\rho(\om_j - \om_k)(x - t)) z_{lk}(t, \rho) \, dt, \quad j,k = \overline{1, n},
\end{equation}
where $b_{jk}$ are defined by \eqref{defb}. The matrix function $w(x, \rho) = z(x, \rho) \exp(\rho B x)$ that is constructed by the solution $z(x, \rho)$ of the system \eqref{intz} is the fundamental matrix of Proposition~\ref{prop:Birk}.

For each fixed $k$, the system \eqref{intz} implies 
\begin{equation} \label{eqzk}
z_k = z_k^0 + V_k z_k,
\end{equation}
where $z_k = z_k(x, \rho)$ is the $k$-th column of $z(x, \rho)$, $z_k^0$ is the $k$-th column of $I$, and $V_k = V_k(\rho)$ is the integral operator given by the right-hand side of \eqref{intz}. The similar relation can be obtained for the system \eqref{syswt}:
\begin{equation} \label{eqzkt}
    \tilde z_k = z_k^0 + \tilde V_k \tilde z_k.
\end{equation}

Subtracting \eqref{eqzkt} from \eqref{eqzk}, we get
$$
\hat z_k = \hat V_k z_k + \tilde V_k \hat z_k,
$$
where $\hat z_k = z_k - \tilde z_k$, $\hat V_k = V_k - \tilde V_k$. Formal calculations show that
\begin{align} \nonumber
    \hat z_k & = \hat V_k z_k + \sum_{\nu = 0}^{\iy} \tilde V_k^{\nu} \hat V_k z_k \\ \label{relzh}
    & = \hat V_k z_k^0 + \hat V_k V_k z_k^0 + \hat V_k V_k^2 z_k + \sum_{\nu = 0}^{\iy} \tilde V_k^{2\nu} (\tilde V_k \hat V_k z_k + \tilde V_k^2 \hat V_k z_k).
\end{align}

It has been proved in \cite{SS20} that
\begin{gather} \label{estzk}
    \| z_k \| \le C, \\
    \label{estVk1}
    \| V_k z_k^0 \|_{L_{\iy}} \le C (\Upsilon(\rho) + |\rho|^{-1}), \\ \label{estVk2}
    \| V_k^2 \|_{L_{\iy} \to L_{\iy}}, \| \tilde V_k^2 \|_{L_{\iy} \to L_{\iy}} \le C (\Upsilon(\rho) + |\rho|^{-1}),
\end{gather}
for $\rho \in \overline{\Gamma}_{\kappa, h}$, $|\rho| \ge \rho^*$. We suppose that $\rho$ belongs to this region everywhere below in this proof. Here and below, the notation $\| . \|_{L_{\iy} \to L_{\iy}}$ is used for the operator norm in the vector space $L_{\mu}(0,1)$.

By virtue of \eqref{eqA}, we have
$$
\hat A(x, \rho) = \sum_{k = d}^{\iy} \rho^{-k} \hat A_k(x).
$$
Hence
\begin{equation} \label{estvh}
\| \hat v_{jl}(., \rho) \|_{L_1} \le C |\rho|^{-1}.
\end{equation}
Using this estimate together with \eqref{defvjkl} and \eqref{intz}, we obtain
\begin{equation} \label{estVk3}
\| \hat V_k \|_{L_{\iy} \to L_{\iy}} \le C |\rho|^{-d}.
\end{equation}

The estimates \eqref{estVk1} and \eqref{estVk2} together imply
\begin{equation} \label{estVk4}
\| \hat V_k V_k z_k^0 \|_{L_{\iy}} \le C |\rho|^{-d} (\Upsilon(\rho) + |\rho|^{-1}).
\end{equation}
By using \eqref{estzk}, \eqref{estVk2}, and \eqref{estVk3}, we get
\begin{equation} \label{estVk5}
\| \hat V_k V_k^2 z_k \|_{L_{\iy}}, \| \tilde V_k^2 \hat V_k z_k \|_{L_{\iy}} \le C |\rho|^{-d} (\Upsilon(\rho) + |\rho|^{-1}).
\end{equation}

It remains to estimate the term $\tilde V_k \hat V_k z_k$. For this purpose, we will show that
\begin{equation} \label{estVk6}
    \| \hat V_k V_k \|_{L_{\iy} \to L_{\iy}} \le C |\rho|^{-d} (\Upsilon(\rho) + |\rho|^{-1}). 
\end{equation}
Let $f$ be an arbitrary vector function of $L_{\iy}(0,1)$ and $g = \tilde V_k \hat V_k f$. In the element-wise form
\begin{multline*}
g_j(x, \rho) = \sum_{l,m = 1}^n \int_{b_{jk}}^x \tilde v_{jl}(t, \rho) \exp(\rho(\om_j - \om_k)(x - t)) \\ \times \int_{b_{lk}}^t \hat v_{lm}(s,\rho) \exp(\rho(\om_l - \om_k)(t-s)) f_m(s) \, ds \, dt.
\end{multline*}
By changing the integration order and taking \eqref{defvjkl}, \eqref{defrhojkl} into account, we derive
$$
g_j(x, \rho) = \sum_{\xi = 1}^n \int_0^1 \left( \sum_{l = 1}^n \hat v_{l\xi}(s, \rho) (\tilde v_{jkl}(s,x,\rho) + \tilde \varrho_{jkl}(s,x,\rho))\right) f_{\xi}(s) \, ds.
$$
By using \eqref{estrho} for $\tilde \varrho_{jkl}$, \eqref{defUps}, and \eqref{estvh}, we obtain the estimate
$$
\max_{x,j} |g_j(x, \rho)| \le C|\rho|^{-d} (\Upsilon(\rho) + |\rho|^{-1}) \max_{\xi,s} |f_{\xi}(s)|,
$$
which yields \eqref{estVk6}.

In view of \eqref{estVk2} and Proposition~\ref{prop:Ups}, one can choose $\rho^*$ so that 
\begin{equation} \label{estVk7}
\| \tilde V_k^2(\rho) \|_{L_{\iy} \to L_{\iy}} \le \frac{1}{2}, \quad |\rho| \ge \rho^*.
\end{equation}

Combining \eqref{relzh}, \eqref{estVk3}--\eqref{estVk7}, we obtain
\begin{equation} \label{estzh}
\| \hat z_k - \hat V_k z_k^0 \|_{L_{\iy}}\le C |\rho|^{-d} (\Upsilon(\rho) + |\rho|^{-1}).
\end{equation}

Now consider the vector function $\eps_k = \hat V_k z_k^0$ with the elements
$$
\eps_{jk}(x, \rho) = \int_{b_{jk}}^x \hat v_{jk}(t, \rho) \exp(\rho (\om_j - \om_k) (x - t)) \, dt
$$
Since 
$$
\max_{j,k} \| \rho^d \hat v_{jk}(., \rho) - \hat a_{d,jk}(., \rho) \|_{L_1} \le C |\rho|^{-1},
$$
we have
$$
\max_{j,x} |\rho^d \eps_{jj}(x, \rho) - \hat a_{d,jj}(x)| \le C |\rho|^{-1}, \quad
\max_{j\ne k; x}|\eps_{jk}(x, \rho)| \le C |\rho|^{-d}(\Upsilon_d(\rho) + |\rho|^{-1}).
$$
Combining the latter estimates with \eqref{estzh}, we obtain \eqref{estEh} for $\hat {\mathcal E}(x, \rho) = \hat z(x, \rho)$.
\end{proof}

Now, we apply the obtained results to equation \eqref{eqv}. Returning from the system \eqref{sysw} back to \eqref{sysu} and then to \eqref{sys} (which is equivalent to \eqref{eqv}), we arrive at Proposition~\ref{prop:y}, which is an immediate corollary of Proposition~\ref{prop:Birk}. For the Mirzoev-Shkalikov case $n = 2m$, $i_{2k+j} = m-k-j$, $j = 0,1$, Proposition~\ref{prop:y} has been obtained in \cite{SS20}.

\begin{prop} \label{prop:y}
For any fixed sector $\Gamma_{\kappa, h}$ and some $\rho^* > 0$, equation \eqref{eqv} has a FSS $\{ y_k(x, \rho) \}_{k = 1}^n$ whose quasi-derivatives $y_k^{[j]}(x, \rho)$, $k = \overline{1,n}$, $j = \overline{0, n-1}$, are continuous for $x \in [0, 1]$, $\rho \in \overline{\mathcal G}$, analytic in $\rho \in \mathcal G$, for each fixed $x \in [0, 1]$, and satisfy the relation
\begin{equation} \label{asympty}
y_k^{[j]}(x, \rho) = (\rho \om_k)^j \exp(\rho \om_k x) (1 + \zeta_{jk}(x, \rho)),
\end{equation}
where
\begin{equation} \label{defzeta}
\max_{j,k,x}|\zeta_{jk}(x, \rho)| \le C(\Upsilon(\rho) + |\rho|^{-1}), \quad \rho \in \overline{\mathcal G},
\end{equation}
and $\Upsilon(\rho)$ is defined by \eqref{defUps1}.
\end{prop}

Consider the differential expressions $\ell_n(y)$ and $\tilde \ell_n(y)$ of form \eqref{defl} with the coefficients $(\sigma_{\nu})$ and $(\tilde \sigma_{\nu})$, respectively, and $i_{\nu} = \tilde i_{\nu}$, $\nu = \overline{0, n-2}$. Suppose that 
\begin{equation} \label{eqsi}
\sigma_{\nu}(x) = \tilde \sigma_{\nu}(x) \:\: \text{a.e. on} \:\: (0, 1), \quad \nu = \overline{\nu_0, n-2}, 
\end{equation}
for a fixed $\nu_0 \in \{ 1, \ldots, n-2 \}$.
Let us study the influence of this condition on the matrices $F(x)$ and $\tilde F(x)$.

According to Definition~\ref{def:F}, the coefficient $\sigma_{\nu}$ influences on the lower diagonal of the matrix $F(x)$ with the index $d_{\nu} = n - 1 - (\nu + i_{\nu})$ and, in some cases, on the diagonals with greater indices.
We mean that the diagonal containing the entry $f_{k,j}$, $k \ge j$, has the index $(k - j)$. That is, the main diagonal has index $0$, the next lower diagonal, index $1$, etc. (see Fig.~\ref{img:indices}). Consequently, the condition \eqref{eqsi} implies that the corresponding diagonals of the matrices $F(x)$ and $\tilde F(x)$ with indices $0$, $1$, \ldots, $(d-1)$ coincide, where
\begin{equation} \label{defd}
    d := n - 1 - \max_{\nu = \overline{0, \nu_0-1}} (\nu + i_{\nu}).
\end{equation}

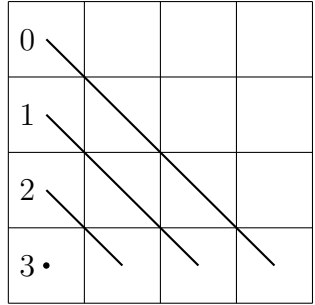
\begin{figure}[h!]
\centering
\begin{tikzpicture}
\draw (0, 0) edge (4, 0);
\draw (0, 0) edge (0, 4);
\draw (4, 0) edge (4, 4);
\draw (0, 4) edge (4, 4);
\draw (0, 1) edge (4, 1);
\draw (0, 2) edge (4, 2);
\draw (0, 3) edge (4, 3);
\draw (1, 0) edge (1, 4);
\draw (2, 0) edge (2, 4);
\draw (3, 0) edge (3, 4);
\draw[thick] (0.5, 3.5) edge (3.5, 0.5);
\draw (0.25, 3.5) node{$0$};
\draw[thick] (0.5, 2.5) edge (2.5, 0.5);
\draw (0.25, 2.5) node{$1$};
\draw[thick] (0.5, 1.5) edge (1.5, 0.5);
\draw (0.25, 1.5) node{$2$};
\filldraw (0.5, 0.5) circle(1pt);
\draw (0.25, 0.5) node{$3$};
\end{tikzpicture}
\caption{Indices of diagonals}
\label{img:indices}
\end{figure}

The $d$-th diagonal of $\hat F(x) = F(x) - \tilde F(x)$ contains linear combinations of $\hat \sigma_{\nu}(x)$ with indices $\nu \in \mathcal N_d$,
\begin{equation*}
    \mathcal N_d = \{ \nu = \overline{0, \nu_0-1} \colon n - 1 - (\nu + i_{\nu}) = d \}.
\end{equation*}

Transforming the systems of form \eqref{sys} to the form \eqref{sysw}, we conclude that \eqref{eqA} holds and $\hat A_d$ depends on $\hat \sigma_{\nu}$ with $\nu \in \mathcal N_d$. More precisely,
$$
\hat A_d(x) = \Omega^{-1} \hat F_d(x) \Omega.
$$
Hence
$$
\hat a_{d,ii}(x) = \frac{1}{n} \sum_{k + j = d} \hat f_{k,j}(x) \om_i^{j-k}, \quad i = \overline{1, n}. 
$$
By using Definition~\ref{def:F}, we obtain the relation
$$
\sum_{k + j = d} \hat f_{k,j}(x) = \sum_{\nu \in \mathcal N_d} S_{\nu} \hat \sigma_{\nu}(x),
$$
where
$$
S_{\nu} = \begin{cases}
                (-1)^{k+1}\sum\limits_{s = 0}^{i_{\nu}} (-1)^s C_{i_{\nu}}^s, \quad \nu = 2k, \\
                (-1)^{k+1} \sum\limits_{s = 0}^{i_{\nu} + 1} (-1)^s  C_{i_{\nu} + 1}^s + 2 (-1)^{k+1} \sum\limits_{s = 0}^{i_{\nu}} (-1)^s C_{i_{\nu}}^s, \quad \nu = 2k+1.
        \end{cases}        
$$
Clearly, $S_{\nu} = 0$ if $i_{\nu} > 0$. Therefore, $\diag(\hat A_d(x))$ is a linear combination of $\sigma_{\nu}(x)$, $\nu \in \mathcal N_d^0$:
\begin{equation*}
    \mathcal N_d^0 = \{ \nu \in \mathcal N_d \colon i_{\nu} = 0 \}.
\end{equation*}

The above arguments allow us to estimate $\hat \zeta_{jk}(x,\rho) = \zeta_{jk}(x, \rho) - \tilde \zeta_{jk}(x,\rho)$ for the remainders from Proposition~\ref{prop:y}, by using Theorem~\ref{thm:Birk}. Thus, we obtain the following corollary.

\begin{cor} \label{cor:y}
Suppose that \eqref{eqsi} holds for the coefficients of the differential expressions $\ell_n(y)$ and $\tilde \ell_n(y)$. Then the difference $\hat \zeta_{jk}(x, \rho)$ of the corresponding Birkhoff-type solution remainders in formula \eqref{asympty} satisfies the estimate
\begin{equation} \label{estzeta}
    \max_x \left|\rho^d \hat \zeta_{jk}(x, \rho) - \int_0^x \hat \theta_{jk}(t) \, dt \right| \le C  (\Upsilon(\rho) + \Upsilon_d(\rho) + |\rho|^{-1}), \quad \rho \in \overline{\mathcal G},
\end{equation}
where $d$ is defined by \eqref{defd}, $\Upsilon(\rho)$ and $\Upsilon_d(\rho)$, by \eqref{defUps} and \eqref{defUpsd}, respectively, and the functions $\hat \theta_{jk}(x)$ are some linear combinations of $\hat \sigma_{\nu}(x)$, $\nu \in \mathcal N_d^0$.
\end{cor}

Note that, if $\hat \sigma_{\nu} \in L_{\mu}(0,1)$ for all $\nu \in \mathcal N_d^0$, then $\hat \theta_{jk} \in L_{\mu}(0,1)$. If $\mathcal N_d^0 = \varnothing$, then $\hat \theta_{jk} = 0$, $j, k = \overline{1, n}$.

\section{Eigenvalue asymptotics} \label{sec:eig}

In this section, we prove Theorems~\ref{thm:asympt} and~\ref{thm:dif} on the eigenvalue asymptotics. The proof of Theorem~\ref{thm:asympt} is based on the standard approach of Naimark \cite{Nai68}. An important difference from the regular case is the usage of the remainder estimates \eqref{estzeta} and the specific properties of the function $\Upsilon(\rho)$ given by Propositions~\ref{prop:Ups},~\ref{prop:nonc}.
In order to prove Theorem~\ref{thm:dif}, we analyze the difference between the eigenvalues of the two boundary value problems $\mathcal L$ and $\tilde {\mathcal L}$ by using the difference of the corresponding Birkhoff-type solutions (Corollary~\ref{cor:y}).

Consider the boundary value problem $\mathcal L$ for equation \eqref{eqv} with the boundary conditions \eqref{bc}. 
For $k = \overline{1, n}$, denote by $C_k(x, \la)$ the solution of equation \eqref{eqv} under the initial conditions $C_k^{[j-1]}(0, \la) = \de_{j,k}$, $j = \overline{1, n}$. The solutions $\{ C_k(x, \la) \}_{k = 1}^n$ form a FSS of \eqref{eqv}.
Therefore, the eigenvalues of the problem $\mathcal L$ coincide with the zeros of the characteristic function $\Delta(\la) = \det [U_s(C_k)]_{s,k = 1}^n$.

\begin{proof}[Proof of Theorem~\ref{thm:asympt}]
\textsc{Step 1. Expansion in the Birkhoff FSS.}
Fix a sector $\Gamma_{\kappa}$ with the property \eqref{order}. Then, for $\rho \in \Gamma_{\kappa, h}$, $|\rho| \ge \rho^*$, equation \eqref{eqv} with $\la = \rho^n$ has a FSS $\{ y_k(x, \rho) \}_{k = 1}^n$ from Proposition~\ref{prop:y}. 

Consider the matrix functions $C(x, \la) := [C_k^{[j-1]}(x, \la)]_{j,k = 1}^n$ and $Y(x, \rho) := [y_k^{[j-1]}(x, \rho)]_{j,k = 1}^n$. Obviously,
\begin{equation} \label{CYA}
C(x, \la) = Y(x, \rho) \mathcal A(\rho),
\end{equation}
where $\mathcal A(\rho)$ is an $(n \times n)$ matrix of coefficients. Consequently,
\begin{equation} \label{relDD}
\Delta(\la) = D(\rho) \det A(\rho), \quad D(\rho) := \det[U_s(y_k)]_{s,k = 1}^n.
\end{equation}

By virtue of Propositions~\ref{prop:Birk} and~\ref{prop:y}, 
\begin{equation} \label{asymptY}
Y(x, \rho) = \diag \{ 1, \rho, \ldots, \rho^{n-1} \} \Omega (I + \mathcal E(x, \rho)) \exp(\rho B x), 
\end{equation}
where $\mathcal E(x, \rho)$ satisfies \eqref{estE}. In particular, Proposition~\ref{prop:Ups} implies $\mathcal E(0, \rho) \to 0$ as $|\rho| \to \iy$, $\rho \in \overline{\Gamma}_{\kappa, h}$.
Consequently, using \eqref{CYA}, \eqref{asymptY}, the initial condition $C(0,\la) = I$, we derive
\begin{equation} \label{asymptA}
\det \mathcal A(\rho) = (\det \Omega)^{-1} \rho^{-n(n-1)/2} (1 + o(1)), \quad |\rho| \to \iy, \quad \rho \in \overline{\Gamma}_{\kappa, h}.
\end{equation}
Hence, for sufficiently large $|\rho|$, we have $\det \mathcal A(\rho) \ne 0$. 

Consider values of $\rho \in \Gamma_{\kappa,h}$ with sufficiently large $|\rho|$. In view of \eqref{relDD},
a number $\la = \rho^n$ is a zero the characteristic function $\Delta(\la)$ if and only if $\rho$ is a zero of $D(\rho)$.

\medskip

\textsc{Step 2. Asymptotics of $D(\rho)$}. Introduce the notation $[1] = 1 + \eps(\rho)$, where $\eps(\rho)$ can denote various functions satisfying
\begin{equation} \label{defeps}
|\eps(\rho)| \le C (\Upsilon(\rho) + |\rho|^{-1}), \quad \rho \in \overline{\mathcal G}.
\end{equation}
Substituting \eqref{asympty} into \eqref{bc} and taking \eqref{estzeta} into account, we obtain
$$
U_s(y_k) = \left\{ \begin{array}{ll} 
     (\rho \om_k)^{p_s}[1], \quad & s \le r, \\
     (\rho \om_k)^{p_s} \exp(\rho \om_k) [1], \quad & s > r.
     \end{array}\right.
$$
Thus
\begin{equation} \label{relDd}
D(\rho) = \rho^p
\begin{vmatrix}
\om_1^{p_1}[1] & \om_2^{p_1}[1] & \dots & \om_n^{p_1}[1] \\
\hdotsfor{4} \\
\om_1^{p_r}[1] & \om_2^{p_r}[1] & \dots & \om_n^{p_r}[1] \\
\om_1^{p_{r+1}}\exp(\rho \om_1)[1] & \om_2^{p_{r+1}} \exp(\rho \om_2)[1] & \dots & \om_n^{p_{r+1}} \exp(\rho \om_n)[1] \\
\hdotsfor{4} \\
\om_1^{p_n} \exp(\rho \om_1)[1] & \om_2^{p_n} \exp(\rho \om_2)[1] & \dots & \om_n^{p_n} \exp(\rho \om_n)[1]
\end{vmatrix}, \quad p := \sum_{s = 1}^n p_s.
\end{equation}

For definiteness, consider the case when $(n - r)$ is even and $\kappa = 1$. The other cases can be treated similarly. It is worth noting that one has to study two neighbouring sectors $\Gamma_{\kappa}$ and $\Gamma_{\kappa + 1}$ in order their images cover the whole $\la$-plane. By analyzing the asymptotics of $D(\rho)$ as $|\rho| \to \iy$, one can show that all its zeros in $\Gamma_{1,h}$ for sufficiently large $|\rho|$ lie in the strip 
\begin{equation} \label{defSR}
\mathcal S_R := \{ \rho \colon \mbox{Re}\, \rho > 0, \, |\mbox{Im}\,\rho| < R \} \subset \Gamma_{1,h},
\end{equation}
if $h$ and $R$ are chosen to be sufficiently large (see Fig.~\ref{img:sectors2}).

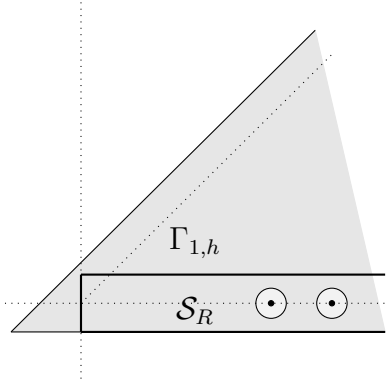
\begin{figure}[h!]
\centering
\begin{tikzpicture}
\filldraw[gray!20] (-0.92, -0.38) -- (4, -0.38) -- (3.08, 3.62) -- cycle;
\filldraw[gray!10] (2.5, 0) circle (0.2);
\filldraw[gray!10] (3.3, 0) circle (0.2);
\draw[dotted] (-1, 0) edge (4, 0);
\draw[dotted] (0, -1) edge (0, 4);
\draw[dotted] (0, 0) edge (3.3, 3.3);
\draw (1.5, 0.8) node{$\Gamma_{1,h}$};
\draw (-0.92, -0.38) edge (4, -0.38);
\draw (-0.92, -0.38) edge (3.08, 3.62);
\draw[thick] (0, -0.38) edge (4, -0.38);
\draw[thick] (0, 0.38) edge (4, 0.38);
\draw[thick] (0, -0.38) edge (0, 0.38);
\draw (1.5, -0.1) node{$\mathcal S_R$};
\filldraw (2.5, 0) circle (1pt);
\filldraw (3.3, 0) circle (1pt);
\draw(2.5, 0) circle (0.2);
\draw (3.3, 0) circle (0.2);
\end{tikzpicture}
\caption{Regions for the proof of Theorem~\ref{thm:asympt}}
\label{img:sectors2}
\end{figure}

For $\rho \in \mathcal S_R$, we have the asymptotics
\begin{gather} \label{relD}
D(\rho) = \rho^p \exp(\rho \om) D_1(\rho), \\ \label{relD1}
D_1(\rho) = D_1^0(\rho) + \eps(\rho), \quad
D_1^0(\rho) := (c_1 - c_2 \exp(\rho (\om_r - \om_{r+1}))), \\ \nonumber
\om := \sum_{k = r + 1}^n \om_k, \quad
c_1 := \det[\om_k^{p_s}]_{s,k = 1}^r \cdot \det[\om_k^{p_s}]_{s,k = r+1}^n \ne 0, \\ \nonumber  c_2 := \det[\om_k^{p_s}]_{s = \overline{1,r}, k = \overline{1, r-1},r+1} \cdot \det[\om_k^{p_s}]_{s = \overline{r+1,n}, k = r, \overline{r+2,n}} \ne 0.
\end{gather}

\medskip

\textsc{Step 3. Asymptotics of zeros for large $|\rho|$}.
Clearly, the zeros of $D_1^0(\rho)$ have the form
$$
\rho_l^0 = \frac{\pi}{\sin\tfrac{\pi r}{n}} (l + \chi), \quad \chi := -\frac{1}{2 \pi i} \log (c_1 / c_2), \quad l \in \mathbb Z.
$$
By using \eqref{relD}, \eqref{relD1} and the standard method based on Rouche's Theorem, we conclude that the zeros of $D_1(\rho)$ and of $D(\rho)$ for sufficiently large $|\rho|$ are simple and have the asymptotics
\begin{equation} \label{smrhol}
\rho_l = \frac{\pi}{\sin\tfrac{\pi r}{n}} (l + \chi + \eps_l), \quad \eps_l = o(1), \quad l \to +\infty.
\end{equation}

Substituting \eqref{smrhol} into the relation $D_1(\rho_l) = 0$ and using \eqref{relD1}, we derive $\eps_l = \eps(\rho_l)$, where $\eps(\rho)$ satisfies \eqref{defeps}. Clearly, $\{ \rho_l \}$ is a non-condensing sequence in $\mathcal G$. Recall that $A_0 \in L_2(0, 1)$. It follows from \eqref{defeps} and Proposition~\ref{prop:nonc} with $\mu = 2$ that $\{ \eps_l \}\in l_2$. 
In this way, we consider the two neighbouring sectors $\Gamma_1$ and $\Gamma_2$.
Returning to the $\la$-plane, we conclude that, in our case, the eigenvalues have the form $\la_l = \rho_l^n$, $l \ge l_0$, where $\rho_l$ satisfy \eqref{smrhol}. It remains to prove that $l_0$ depends only on $n$, $r$, and $(p_s)_{s = 1}^n$.

\medskip

\textsc{Step 4. Estimate of $|\Delta(\la)|$ from below}.
Consider the region
$$
\mathcal G_{\de} := \{ \rho \in \mathcal G \colon |\rho - \rho_l| \ge \de, \, l \ge l_0 \}, \quad \de > 0.
$$
Using \eqref{relDd}--\eqref{smrhol}, we obtain the estimate
\begin{equation} \label{Db}
|D(\rho)| \ge C_{\de} |\rho|^p \exp (\mbox{Re}\,(\rho \om)), \quad \rho \in \mathcal G_{\de}, 
\end{equation}
where $C_{\de}$ is a constant depending on $\de$ and $\mathcal G$.
By using \eqref{asymptA} and \eqref{Db}, we get
\begin{equation} \label{Deltab}
|\Delta(\rho^n)| \ge C_{\de} |\rho|^{p - n(n-1)/2} \exp (\mbox{Re}\,(\rho \om)), \quad \rho \in \mathcal G_{\de}.
\end{equation}

\medskip

\textsc{Step 5. Difference $(\Delta(\la) - \Delta^0(\la))$.}
Consider the problem $\mathcal L^0$ of the same form as $\mathcal L$ with the zero coefficients $\sigma_{\nu}^0 = 0$, $\nu = \overline{0, n-2}$, $u_{s,j}^0 = 0$, $j = \overline{1, p_s}$, $s = \overline{1,n}$. Let $\Delta^0(\lambda)$ be the characteristic function of $\mathcal L^0$. By using the formulas \eqref{relDD} and \eqref{relDd}, we obtain
\begin{equation} \label{DD0}
\Delta(\rho^n) - \Delta^0(\rho^n) = o\left( \rho^{p - n(n-1)/2} \exp(\rho \om) \right), \quad \rho \in \overline{\Gamma}_{\kappa,h}, \quad |\rho| \to \iy.
\end{equation}
Combining \eqref{Deltab} and \eqref{DD0}, we get
\begin{equation} \label{ineqD}
|\Delta(\la) - \Delta^0(\la)| < |\Delta(\la)|, 
\end{equation}
for $\la = \rho^n$, $\rho \in \mathcal G_{\de}$, and sufficiently large $|\rho|$. Clearly, the inequality \eqref{ineqD} can be obtained for the two neighbouring sectors $\Gamma_{\kappa,h}$ and $\Gamma_{\kappa + 1,h}$, whose images cover the whole $\la$-plane. Consequently, \eqref{ineqD} holds on some contour $\{ \la \colon |\la| = R \}$ with a sufficiently large $R$. Since the functions $\Delta(\la)$ and $\Delta^0(\la)$ are entire in $\la$, then, by virtue of Rouche's Theorem, these two functions have the same number of zeros in the circle $\{ \la \colon |\la| < R \}$. Thus, the numeration of the zeros $\{ \la_l \}$ and $\{ \la_l^0 \}$ of $\Delta(\la)$ and $\Delta^0(\la)$, respectively, starts from the same index $l_0$. The shift of numeration leads to \eqref{asymptla}.
\end{proof}

Further, we need the following technical lemma.

\begin{lem} \label{lem:Ddot}
Let $D_1(\rho)$ be a function of form \eqref{relD1}, where $\eps(\rho)$ is analytic function in $\mathcal G$ satisfying \eqref{defeps}, and $\{ \rho_l \}_{l \ge l_0} \subset \mathcal G$ be an arbitrary sequence of form \eqref{smrhol} with $\{ \eps_l \} \in l_2$. Then
$$
\dot D_1(\rho_l) = c + \varkappa_l, \quad c \in \mathbb C, \quad \{ \varkappa_l \} \in l_2,
$$
where $\dot D_1(\rho) = \tfrac{d}{d\rho} D_1(\rho)$.
\end{lem}

\begin{proof}
It follows from \eqref{relD1} that
$$
\dot D_1(\rho_l) = \dot D_1^0(\rho_l) + \dot \eps(\rho_l).
$$
Obviously, $\dot D_1(\rho_l) = c + \varkappa_l$. The Cauchy formula yields
$$
\dot \eps(\rho_l) = \frac{1}{2\pi i} \int\limits_{|\rho - \rho_l| = \de} \frac{\eps(\rho)}{(\rho - \rho_l)^2} \, d\rho,
$$
where $\de > 0$ is so small that $\{ \rho \colon |\rho - \rho_l| = \de \} \subset \mathcal G$. Hence
$$
|\dot \eps(\rho_l)| \le \de^{-1} \max\limits_{|\rho - \rho_l| = \de} |\eps(\rho)|.
$$
Denote by $\{ \rho_l^{\diamond} \}$ the points such that
$$
|\rho_l^{\diamond} - \rho_l| = \de, \quad
|\eps(\rho_l^{\diamond})| = \max\limits_{|\rho - \rho_l| = \de} |\eps(\rho)|. 
$$
Clearly, $\{ \rho_l^{\diamond} \}_{l \ge l_0}$ is a non-condensing sequence in $\mathcal G$. Therefore, it follows from \eqref{defeps} and Proposition~\ref{prop:nonc} that $\{ \eps(\rho_l^{\diamond}) \} \in l_2$. Consequently, $\{ \dot \eps(\rho_l) \} \in l_2$. This completes the proof.
\end{proof}

\begin{proof}[Proof of Theorem~\ref{thm:dif}]
Suppose that the problems $\mathcal L$ and $\tilde{\mathcal L}$ satisfy the conditions of Theorem~\ref{thm:dif}, that is, $\sigma_{\nu}(x) = \tilde \sigma_{\nu}(x)$ a.e. on $(0,1)$, $\nu = \overline{\nu_0, n-2}$, $u_{s,p_s-j} = \tilde u_{s,p_s -j}$, $j = \overline{0, d-2}$, $d := n - 1 - \max\limits_{\nu = \overline{0, \nu_0-1}} (\nu + i_{\nu})$. By virtue of Theorem~\ref{thm:asympt}, the eigenvalues have the form $\la_l = (-1)^{n-r} \rho_l^n$ and $\tilde \la_l = (-1)^{n-r} \tilde \rho_l^n$, $l \ge 1$, where $\rho_l$ and $\tilde \rho_l$ have the asymptotics \eqref{smrhol} with $\chi = \tilde \chi$.
For definiteness, consider the case of even $(n-r)$ and $\kappa = 1$. According to the proof of Theorem~\ref{thm:asympt}, the numbers $\rho_l$ and $\tilde \rho_l$ for sufficiently large $l$ are the zeros of the functions $D_1(\rho)$ and $\tilde D_1(\rho)$, respectively, defined by \eqref{relD1}. In order to estimate $\hat \rho_l$, we analyze the difference $\hat D_1(\rho)$.

Using the conditions of Theorem~\ref{thm:dif},  Corollary~\ref{cor:y}, and \eqref{bc}, we obtain
\begin{equation} \label{difU}
U_s(y_k) - \tilde U_s(\tilde y_k) = \left\{ \begin{array}{ll}
                                        (\rho \om_k)^{p_s} \rho^{-d} (\hat c_{sk} + \hat \eps(\rho)), \quad & s \le r, \\
                                        (\rho \om_k)^{p_s}\exp(\rho \om_k)\rho^{-d} (\hat c_{sk} + \hat \eps(\rho)), \quad & s > r.
                                    \end{array} \right.
\end{equation}
Here and below in this proof, we denote by $\hat \eps(\rho)$ various functions satisfying
$$
|\hat \eps(\rho)| \le C (\Upsilon(\rho) + \Upsilon_d(\rho) + |\rho|^{-1}),
$$
and by $\hat c$ with and without indices constants depending on the values
\begin{equation} \label{thu}
\int_0^1 \hat \theta_{jk}(t) \, dt, \quad j = \overline{0,n-1}, \, k = \overline{1, n}, \quad \text{and} \quad \hat u_{s,p_s-d+1}, \quad s = \overline{1, n},
\end{equation}
where $\theta_{jk}(t)$ are the functions from \eqref{estzeta}.

Repeating the arguments of Step~2 in the proof of Theorem~\ref{thm:asympt}, we obtain
\begin{equation} \label{D1h}
\hat D_1(\rho)= \rho^{-d} (\hat c_1 - \hat c_2 \exp(\rho(\om_r - \om_{r+1})) + \hat \eps(\rho)), \quad \rho \in \mathcal S_R,
\end{equation}
for sufficiently large $|\rho|$.

It follows from $D_1(\rho_l) = 0$ and $\tilde D_1(\tilde \rho_l) = 0$ that
\begin{equation} \label{sm1}
D_1(\rho_l) - D_1(\tilde \rho_l) + \hat D_1(\tilde \rho_l) = 0.
\end{equation}
The complex Taylor formula implies
\begin{equation} \label{sm2}
D_1(\tilde \rho_l) - D_1(\rho_l) =  \dot D_1(\rho_l)(\tilde \rho_l - \rho_l) + \mathcal R(\rho_l, \tilde \rho_l),
\end{equation}
where 
$$
\mathcal R(\rho_l, \tilde \rho_l) = \frac{(\tilde \rho_l - \rho_l)^2}{2\pi i} \int\limits_{|\rho - \rho_l| = \de} \frac{D_1(w) \, dw}{(w - \rho_l)^2 (w - \tilde \rho_l)}.
$$
Using \eqref{relD1} and \eqref{smrhol}, one can easily show that
$$
|\mathcal R(\rho_l, \tilde \rho_l)| \le C |\rho_l - \tilde \rho_l|^2,
$$
for sufficiently large $l$. The latter estimate together with \eqref{sm1}, and \eqref{sm2} imply
\begin{equation} \label{sm3}
\rho_l - \tilde \rho_l = \hat D_1(\tilde \rho_l) (\dot D_1(\rho_l) + O(\rho_l - \tilde \rho_l)).
\end{equation}

Let us estimate the right-hand side of \eqref{sm3}.
According to Proposition~\ref{prop:Ups}, $\hat \eps(\rho) \to 0$ as $|\rho| \to \iy$, $\rho \in \mathcal S_R$. Therefore, \eqref{D1h} and \eqref{smrhol} imply 
\begin{equation} \label{D1rho}
\hat D_1(\tilde \rho_l) = l^{-d} (\hat c_0 + \eta_l), \quad \eta_l = o(1), \quad l \to \iy.
\end{equation}

Using \eqref{sm3}, \eqref{D1rho}, and Lemma~\ref{lem:Ddot}, we derive the asymptotics
$$
\rho_l - \tilde \rho_l = l^{-d} (\hat c + \de_l), \quad \de_l = o(1), \quad l \to \iy.
$$

Consider the following special cases.

\smallskip

(i) Suppose that $\int_0^1 \hat \sigma_{\nu}(x) \, dx = 0$, $\nu \in \mathcal N_d^0$, and $\hat u_{s,p_s-d+1} = 0$, $s = \overline{1, n}$. Then, according to Corollary~\ref{cor:y}, all the values \eqref{thu} equal zeros, so $\hat c = 0$.

\smallskip

(ii) If $\sigma_{\nu} \in L_2(0,1)$, $\nu \in \mathcal N_d$, then
$A_d \in L_2(0,1)$. Applying Proposition~\ref{prop:nonc} to $\Upsilon(\rho)$ and $\Upsilon_d(\rho)$, we obtain $\{ \de_l \} \in l_2$.

\smallskip

The proof is complete.
\end{proof}

\begin{proof}[Proof of Corollary~\ref{cor:odd}]
Let us apply Theorem~\ref{thm:dif} to an arbitrary problem $\mathcal L$ with $n = 2m+1$, $\tilde{\mathcal L} = \mathcal L^0$ and $\nu_0 = n-1$. The inequality $i_{2k+j} \le m - k - j$ implies that the minimal value of $d$ equals $1$. In other words, the main diagonal in the corresponding matrix $F(x)$ always equal zero in the odd case. We have $i_{n-2} = 0$ and $i_{n-3} \in \{ 0, 1 \}$. For $i_{n-3} = 0$, we have $\mathcal N_d = \mathcal N_d^0 = \{ (n - 2)\}$, and for $i_{n-3} = 1$, $\mathcal N_d = \{ (n-3), (n-2) \}$, $\mathcal N_d^0 = \{ (n-2) \}$. Thus, Theorem~\ref{thm:dif} immediately yields the claim.
\end{proof}

\section{Examples} \label{sec:examp}

This section illustrates the application of Theorems~\ref{thm:asympt} and~\ref{thm:dif} to various classes of differential operators with distribution coefficients.

\begin{example}
Suppose that $n = 3$, $i_0 = 1$, $i_1 = 0$, $\sigma_{\nu} \in L_1(0,1)$, $\nu = 0, 1$. Then the differential expression \eqref{defl} takes the form
$$
\ell_3(y) = y^{(3)} - (\sigma_1(x) y)' - \sigma_1(x) y' - \sigma'_0(x) y, \quad x \in (0,1),
$$
and the associated matrix equals
\begin{equation*} 
F(x) = \begin{bmatrix}
            0 & 1 & 0 \\
            (\sigma_0 + \sigma_1) & 0 & 1 \\
            0 & -(\sigma_0 - \sigma_1) & 0
        \end{bmatrix}.
\end{equation*}
Clearly, 
\begin{equation} \label{F3}
F_0(x) \equiv 0, \quad
F_1(x) = \begin{bmatrix}
            0 & 0 & 0 \\
            (\sigma_0 + \sigma_1) & 0 & 0 \\
            0 & -(\sigma_0 - \sigma_1) & 0
        \end{bmatrix}.
\end{equation}

Consider the sector $\Gamma_1 = \{ \rho \colon 0 < \arg \rho < \pi/3 \}$. Then $\om_1 = \exp(-2\pi \mathrm{i}/3)$, $\om_2 = \exp(2\pi \mathrm{i}/3)$, $\om_3 = 1$. It follows from \eqref{F3} that $A_0(x) \equiv 0$ and
$$
\diag(A_1(x)) = \diag(\Omega^{-1} F_1(x) \Omega) = \frac{2}{3} \sigma_1(x) \begin{bmatrix}
            \om_1^{-1} & 0 & 0 \\
            0 & \om_2^{-1} & 0 \\
            0 & 0 & \om_3^{-1}
        \end{bmatrix}.
$$

Applying Theorem~\ref{thm:Birk} to the systems \eqref{sysw} and \eqref{syswt} with $\tilde A(x,\rho) \equiv 0$ and $d = 1$, we obtain the following asymptotics for the fundamental solution matrix:
$$
w(x, \rho) = \biggl( I + \frac{2}{3 \rho} \int_0^x \sigma_1(t) \, dt \, B^{-1} +  \frac{\gamma(x, \rho)}{\rho}\biggr) \exp(\rho B x),
$$
where 
\begin{equation} \label{estga}
\max_x \| \gamma(x, \rho) \| \le C (\Upsilon(\rho) + \Upsilon_1(\rho) + |\rho|^{-1}), \quad \rho \in \overline{\mathcal G}.
\end{equation}
Passing to the FSS $\{ y_k(x, \rho) \}_{k = 1}^3$ of the equation $\ell_3(y) = \rho^n y$, we obtain the asymptotics
\begin{equation} \label{asympty3}
y_k^{[j]}(x, \rho) = (\rho \om_k)^j \exp(\rho \om_k x) \biggl( 1 + \frac{2}{3 \rho \om_k} \int_0^x \sigma_1(t) \, dt + \frac{\gamma_{jk}(x, \rho)}{\rho}\biggr), \quad k = \overline{1,n}, \, j = \overline{0,n-1},
\end{equation}
where scalar functions $\gamma_{jk}(x,\rho)$ satisfy the same estimate \eqref{estga} as the matrix function $\gamma(x, \rho)$.

Consider the differential equation $\ell_3(y) = \la y$ with the following boundary conditions $(r = 1)$:
$$
y(0) = 0, \quad y(1) = 0, \quad y^{[1]}(1) = 0.
$$
By virtue of Theorem~\ref{thm:asympt} and Corollary~\ref{cor:odd}, the eigenvalues of this problem have the asymptotics
\begin{equation} \label{asymptla3}
\la_l = \left( \frac{2\pi}{\sqrt 3} \biggl( l + \chi + \frac{\chi_1}{l} +  \frac{\delta_l}{l}\biggr)\right)^n, \quad \de_l = o(1), \quad l \to \iy,
\end{equation}
where $\chi_1$ depends on $\int_0^1 \sigma_1(x) \, dx$ and, if $\sigma_{\nu} \in L_2(0,1)$, $\nu = 0,1$, then $\{ \de_l \} \in l_2$.

The function $D(\rho)$ defined by \eqref{relDD} has the form
\begin{equation} \label{D3}
D(\rho) = \begin{vmatrix}
            y_1(0,\rho) & y_2(0, \rho) & y_3(0, \rho) \\
            y_1(1,\rho) & y_2(1,\rho) & y_3(1,\rho) \\
            y_1^{[1]}(1,\rho) & y_2^{[1]}(1,\rho) & y_3^{[1]}(1,\rho)
          \end{vmatrix}.
\end{equation}
Substituting the asymptotics \eqref{asympty3} into \eqref{D3} and finding the asymptotics of the zeros $\{ \rho_l \}$ of $D(\rho)$ in the strip $S_R$, we obtain the values 
$$
\chi = \frac{1}{6}, \quad \chi_1 = \frac{1}{\pi^2} \int_0^1 \sigma_1(x) \, dx
$$
of the constants in \eqref{asymptla3}.
\end{example}

In the next examples, for the sake of simplicity, we assume that $u_{s,j} = \tilde u_{s,j}$, $s = \overline{1,n}$, $j = \overline{1,p_s}$.

\begin{example} \label{ex:2}
Suppose that $n = 2m$, $i_{\nu} = 0$, $\sigma_{\nu} \in L_1(0,1)$, $\nu = \overline{0,n-2}$. Due to Definition~\ref{def:F}, the entries of the associated matrix $F(x) = [f_{k,j}(x)]_{k,j = 1}^n$ are given by the relations
\begin{align*}
& f_{n-k,k+1} = (-1)^{k+1} \sigma_{2k}, \quad k = \overline{0,m-1}, \\
& f_{n-k-1,k+1} = f_{n-k, k+2} = (-1)^k \sigma_{2k+1}, \quad k = \overline{0, m-2},
\end{align*}
and all the other entries are defined as $f_{k,j} = \de_{k+1,j}$. For instance,
$$
\ell_6(y) = y^{(6)} + (\sigma_4 y'')'' + [(\sigma_3 y'')' + (\sigma_3 y')''] - (\sigma_2 y')' - [(\sigma_1 y)' + \sigma_1 y'] + \sigma_0 y,
$$
and the associated matrix is
$$
\begin{tikzpicture}
     \matrix (mat) [matrix of nodes, left delimiter={[},right delimiter={]}]
      {%
        \:0\: & \:1\: & \:0\: & \:\:0\:\: & \:\:0\:\: & \:\:0\:\: \\
        0 & 0 & 1 & 0 & 0 & 0 \\
        0 & 0 & 0 & 1 & 0 & 0 \\
        0 & $-\sigma_3$ & $-\sigma_4$ & 0 & 1 & 0 \\
        $\sigma_1$ & $\sigma_2$ & $-\sigma_3$ & 0 & 0 & 1 \\
        $-\sigma_0$ & $\sigma_1$ & 0 & 0 & 0 & 0 \\
      };
    % do the strike out thing
      \draw[gray] (mat-2-1.center)  -- (mat-6-5.center);
      \draw[gray] (mat-3-1.center) -- (mat-6-4.center);
      \draw[gray] (mat-4-1.center) -- (mat-6-3.center);
      \draw[gray] (mat-5-1.center) -- (mat-6-2.center);
\end{tikzpicture}
$$

Clearly, for each $d = \overline{1,n-1}$, the $d$-th diagonal contains only $\sigma_{n-d-1}$. Since the main diagonal is zero, the remainder term $\eps_l$ of the asymptotics \eqref{asymptla} has the form \eqref{asymptodd}, similarly to the case of odd $n$.

Consider problems $\mathcal L$ and $\tilde{\mathcal L}$ such that $\sigma_{\nu}(x) = \tilde \sigma_{\nu}(x)$ a.e. on $(0,1)$ for $\nu = \overline{\nu_0, n-2}$, $\nu_0 \in \{ 1, \ldots, n-1 \}$. Then, in Theorem~\ref{thm:dif}, $d = n-\nu_0$, $\mathcal N_d = \mathcal N_d^0 = \{ (\nu_0-1) \}$. Hence 
$$
\hat \rho_l = l^{-(n - \nu_0)}(\hat c + \delta_l), \quad \delta_l = o(1),
$$
and the constant $\hat c$ linearly depends on $\int\limits_0^1 \hat \sigma_{\nu_0-1}(x) \, dx$. In addition, if $\hat \sigma_{\nu_0-1} \in L_2(0,1)$, then $\{ \delta_l \} \in l_2$.
\end{example}

\begin{example}
Suppose that $n = 2m$, $i_{\nu} = 1$, $\sigma_{\nu} \in L_2(0,1)$, $\nu = \overline{0,n-2}$, $\nu_0 \in \{ 1, \ldots, n-1\}$. For instance,
$$
\ell_4(y) = y^{(4)} + (\sigma_2' y')' + [(\sigma_1' y)' + \sigma_1' y'] - \sigma_0' y
$$
and the associated matrix equals
$$
F(x) = 
\begin{bmatrix}
        0 & 1 & 0 & 0 \\
        -\sigma_1 & -\sigma_2 & 1 & 0 \\
        (\sigma_0 - \sigma_1 \sigma_2) & -\sigma_2^2 & \sigma_2 & 1 \\
        -\sigma_1^2 & (-\sigma_0 - \sigma_1 \sigma_2) & \sigma_1 & 0
    \end{bmatrix}.
$$

Suppose that, for the problems $\mathcal L$ and $\tilde{\mathcal L}$, we have $\sigma_{\nu}(x) = \tilde \sigma_{\nu}(x)$ a.e. on $(0,1)$ for $\nu = \overline{\nu_0, n-2}$. Then $d = n - \nu_0 - 1$, $\mathcal N_d = \{ (\nu_0 - 1) \}$, $\mathcal N_d^0 = \varnothing$. Therefore, Theorem~\ref{thm:dif} implies $\hat \rho_l = l^{-(n-\nu_0-1)} \varkappa_l$, $\{ \varkappa_l \} \in l_2$.
\end{example}

\begin{example} \label{ex:4}
Consider the case of Mirzoev and Shkalikov \cite{MS16}: $n = 2m$, $i_{2k+j} = m-k-j$, $j \in \{ 0, 1 \}$, $\sigma_{\nu} \in L_2(0,1)$, $\nu = \overline{0,n-2}$. The structure of the associated matrix $F(x)$ is provided in \cite{MS16}. Suppose that for the problems $\mathcal L$ and $\tilde{\mathcal L}$, we have $\sigma_{\nu}(x) = \tilde \sigma_{\nu}(x)$ a.e. on $(0,1)$ for $\nu = \overline{2\nu_1, n-2}$. Then $d = m-\nu_1$, $\mathcal N_d = \{ (2\nu_1-2), (2\nu_1-1)\}$, $\mathcal N_d^0 = \varnothing$. Hence, Theorem~\ref{thm:dif} implies $\hat \rho_l = l^{-(m - \nu_1)} \varkappa_l$, $\{ \varkappa_l \} \in l_2$.
\end{example}

The cases similar to Examples~\ref{ex:2}--\ref{ex:4} can be considered for odd $n$.

\section{Asymptotics of weight numbers} \label{sec:weight}

In this section, we define the weight numbers $\{ \beta_l \}$ and obtain for them results analogous to Theorems~\ref{thm:asympt} and~\ref{thm:dif} for the eigenvalues.

Together with $U_s(y)$, $s = \overline{1, n}$, consider the linear form
$$
U_0(y) = y^{[p_0]}(0) + \sum_{j = 1}^{p_0} u_{0,j} y^{[j-1]}(0), \quad p_0 \ne p_s, \: s = \overline{1, r}.
$$
Denote by $\mathcal L^{\bullet}$ the boundary value problem for equation \eqref{eqv} with the boundary conditions $U_s(y) = 0$, $s = \overline{0, n} \setminus r$. The eigenvalues of $\mathcal L^{\bullet}$ coincide with the zeros of the characteristic functions $\Delta^{\bullet}(\la) := \det[U_s(C_k)]_{s = \overline{0,n}\setminus r, \, k = \overline{1,n}}$.

Define the weight numbers $\{ \beta_l \}$ as follows:
$$
\beta_l := \Res_{\la = \la_l} \frac{\Delta^{\bullet}(\la)}{\Delta(\la)}.
$$

By Theorem~\ref{thm:dif},  for sufficiently large $l$, the eigenvalues $\{ \la_l \}$ of the problem $\mathcal L$ are simple. Therefore,
\begin{equation} \label{bel}
\beta_l = \frac{\Delta^{\bullet}(\la_l)}{\tfrac{d}{d\la} \Delta(\la_l)}
\end{equation}
for such values of $l$. It is worth considering the weight numbers only for sufficiently large indices $l$.

\begin{example}
Let $n = 2$, $r = 1$, $p_1 = p_2 = 0$, $p_0 = 1$, $u_{0,1} = 0$. Then $\mathcal L$ and $\mathcal L^{\bullet}$ are the boundary value problems for the Sturm-Liouville equation 
\begin{equation} \label{StL}
y'' - q(x) y = \la y, \quad x \in (0,1), 
\end{equation}
with the boundary conditions $y(0) = y(1) = 0$ and $y^{[1]}(0) = y(1) = 0$, respectively. Hence 
$$
\beta_l = \frac{C_1(1, \la_l)}{\tfrac{d}{d\la} C_2(1, \la_l)},
$$
where $C_k(x, \la)$ are the solutions of \eqref{StL} under the initial conditions $C_k^{[j-1]}(0, \la) = \de_{j,k}$, $j,k = 1, 2$. One can easily show that $C_2(x, \la_l)$ are the eigenfunctions of $\mathcal L$ and 
\begin{equation} \label{all}
\beta_l = -\al_l^{-1}, \quad \al_l := \int\limits_0^1 C_2^2(x, \la_l) \, dx.
\end{equation}
For a real-valued potential $q \in L_2(0,1)$, the numbers $\{ \la_l, \al_l \}_{l \ge 1}$ are the classical spectral data of the inverse Sturm-Liouville problem (see, e.g., \cite{Mar77, FY01}). For the case of complex-valued $q \in L_2(0,1)$, the so-called generalized spectral data have been introduced in \cite{But07}. In the Dirichlet-Dirichlet case, the generalized weight numbers coincide with $\al_l$ defined by \eqref{all} for sufficiently large $l$ (see \cite{BSY13}).
\end{example}

\begin{thm} \label{thm:w1}
For sufficiently large $l$, the following relation holds:
\begin{equation} \label{asymptbeta}
\beta_l = l^{n - 1 + p_0 - p_r} (\beta^0 + \varkappa_l), \quad 
\{ \varkappa_l \} \in l_2,
\end{equation}
where the constant $\beta^0$ depends only on $n$, $r$, and $(p_s)_{s = 0}^n$.
\end{thm}

\begin{proof}
For definiteness, consider the case of even $(n-r)$. Recall that the eigenvalues of $\mathcal L$ have the form $\la_l = \rho_l^n$, where $\{ \rho_l \}$ for sufficiently large $l$ belong to $\mathcal S_R$ and fulfill \eqref{smrhol}.

Similarly to the proof of Theorem~\ref{thm:asympt}, we obtain the formulas for $\rho \in \mathcal S_R$:
\begin{gather} \label{relDDb}
\Delta^{\bullet}(\la) = D^{\bullet}(\rho) \det \mathcal A(\rho), \quad D^{\bullet}(\rho) = \det[U_s(C_k)]_{s = \overline{0,n}\setminus r, \, k = \overline{1, n}},  \\ \label{relDb}
D^{\bullet}(\rho) = \rho^{p - p_r + p_0} \exp(\rho \om) D^{\bullet}_1(\rho), \\ \label{relD1b}
D^{\bullet}_1(\rho) = D^{\bullet,0}_1(\rho) + \eps(\rho), \quad
D^{\bullet,0}_1(\rho) = c^{\bullet}_1 - c^{\bullet}_2 \exp(\rho(\om_r - \om_{r+1})),
\end{gather}
where $\eps(\rho)$ is a function satisfying \eqref{defeps}, not necessarily equal to $\eps(\rho)$ in \eqref{relD1}, and $c^{\bullet}_1$, $c^{\bullet}_2$ are some constants different from $c_1$, $c_2$. Substituting \eqref{relDD}, \eqref{relD}, \eqref{relDDb}, \eqref{relD1b} into \eqref{bel}, we derive
\begin{equation} \label{beD1}
\beta_l = \frac{n \rho_l^{n-1} D^{\bullet}(\rho_l)}{\frac{d}{d\rho} D(\rho_l)} = n \rho_l^{n - 1 + p_0 - p_r} \frac{D^{\bullet}_1(\rho_l)}{\frac{d}{d\rho} D_1(\rho_l)}.
\end{equation}
Using \eqref{relD1}, \eqref{smrhol}, Lemma~\ref{lem:Ddot} for $D_1^{\bullet}(\rho_l)$ and taking into account that $\{ \eps(\rho_l) \} \in l_2$, we obtain 
$$
D^{\bullet}_1(\rho_l) = s_1 + \varkappa_{l,1}, \quad 
\tfrac{d}{d\rho} D_1(\rho_l) = s_2 + \varkappa_{l,2}, 
$$
where $s_1$, $s_2 \ne 0$ are constants and $\{ \varkappa_{l,1} \}, \{ \varkappa_{l,2} \} \in l_2$. Hence, we arrive at \eqref{asymptbeta}.
\end{proof}

\begin{remark}
Theorems~\ref{thm:asympt} and \ref{thm:w1} are valid for the eigenvalues and the weight numbers, respectively, of the boundary value problems for the system $\vec y' = (F(x) + \Lambda) \vec y$ with the appropriate boundary conditions generated by the linear forms $U_s(y)$, $s = \overline{0,n}$, and with an arbitrary matrix function $F(x) = [f_{k,j}]_{k,j=1}^n$ (not necessarily related to the differential expression $\ell_n(y)$) satisfying the conditions:
\begin{gather*}
    f_{k,j} = \de_{k+1,j}, \: k < j, \quad f_{k,k} \in L_2(0,1), \quad f_{k,j} \in L_1(0,1), \: k > j, \quad \mbox{trace}(F(x)) \equiv 0.
\end{gather*}
\end{remark}

Further, we formulate and prove the theorem for the weight numbers $\{ \beta_l \}$ similar to Theorem~\ref{thm:dif} for the eigenvalues.

\begin{thm} \label{thm:w2}
Suppose that $\sigma_{\nu}(x) = \tilde \sigma_{\nu}(x)$ for a.e. $x \in (0,1)$, $\nu = \overline{\nu_0, n-2}$, and $u_{s,p_s-j} = \tilde u_{s,p_s-j}$,$j = \overline{0,d-2}$, $s = \overline{0,n}$, where $d$ is defined by \eqref{defd1}.
Then
$$
\beta_l - \tilde \beta_l = l^{n - 1 + p_0 - p_r - d} (\hat c + \delta_l), \quad \delta_l = o(1), \quad l \to \iy, 
$$
where $\hat c$ and $\delta_l$ have the properties similar to the ones in Theorem~\ref{thm:dif}, where $s = \overline{1,n}$ in \eqref{valc} is replaced with $s = \overline{0,n}$.
\end{thm}

The notations $\hat c$ and $\de_l$ in Theorems~\ref{thm:dif} and~\ref{thm:w2} are used for different values. However, we use the same notations in the both cases to emphasize the similar remainder properties.

\begin{proof}[Proof of Theorem~\ref{thm:w2}]
Assume that the conditions of the theorem holds for $\mathcal L$ and $\tilde {\mathcal L}$.
Similarly to the previous proofs, we consider the case of even $(n - r)$. The odd case is analogous. Let us use the formula \eqref{beD1} for $\be_l$ and $\tilde \be_l$. For shortness, put $q := n-1+p_0-p_r$. Then
\begin{align} \nonumber
\frac{\beta_l - \tilde \beta_l}{n} & = 
\rho_l^q \frac{D^{\bullet}_1(\rho_l)}{\tfrac{d}{d\rho} D_1(\rho_l)} - \tilde \rho_l^q \frac{\tilde D^{\bullet}_1(\tilde \rho_l)}{\tfrac{d}{d\rho} \tilde D_1(\tilde \rho_l)} \\ \label{dbe} & =
\rho_l^q \frac{D^{\bullet}_1(\rho_l) - \tilde D^{\bullet}_1(\tilde \rho_l)}{\tfrac{d}{d\rho} D_1(\rho_l)} + \rho_l^q \tilde D^{\bullet}_1(\rho_l) \frac{\tfrac{d}{d\rho} \tilde D_1(\tilde \rho_l) - \tfrac{d}{d\rho} D_1(\rho_l)}{\tfrac{d}{d\rho} D_1(\rho_l) \tfrac{d}{d\rho} \tilde D_1(\tilde \rho_l)} + (\rho_l^q - \tilde \rho_l^q) \frac{D^{\bullet}_1(\rho_l)}{\tfrac{d}{d\rho} D_1(\rho_l)}
\end{align}
for sufficiently large $l$.

Consider the first term in \eqref{dbe}. Obviously,
$$
D_1^{\bullet}(\rho_l) - \tilde D_1^{\bullet}(\tilde \rho_l) = D_1^{\bullet}(\rho_l) - D_1^{\bullet}(\tilde \rho_l) + \hat D_1^{\bullet}(\tilde \rho_l).
$$
For $D_1^{\bullet}(\tilde \rho_l)$, the asymptotics similar to \eqref{D1rho} holds. Using the Taylor formula, we obtain
\begin{gather*}
D_1^{\bullet}(\tilde \rho_l) - D_1^{\bullet}(\rho_l) = \dot D_1^{\bullet}(\rho_l) (\tilde \rho_l - \rho_l) + \mathcal R^{\bullet}(\rho_l, \tilde \rho_l), \\ |\mathcal R^{\bullet}(\rho_l, \tilde \rho_l)| \le C |\rho_l - \tilde \rho_l|^2.
\end{gather*}
Thus, similarly to the proof of Theorem~\ref{thm:dif}, we obtain
$$
D_1^{\bullet}(\rho_l) - \tilde D_1^{\bullet}(\tilde \rho_l) = l^{-d} (\hat c + \delta_l), \quad \delta_l = o(1), \quad l \to \iy.
$$
Using \eqref{smrhol} and Lemma~\ref{lem:Ddot}, we arrive at the asymptotics $l^{q-d} (\hat c + \delta_l)$ for the first term of \eqref{dbe}. For the second and the third terms, the same asymptotics can be obtained analogously. This completes the proof of Theorem~\ref{thm:w2}.
\end{proof}

\medskip

{\bf Funding.} This work was supported by Grant 21-71-10001 of the Russian Science Foundation, https://rscf.ru/en/project/21-71-10001/.

\noindent Natalia Pavlovna Bondarenko \\
1. Department of Applied Mathematics and Physics, Samara National Research University, \\
Moskovskoye Shosse 34, Samara 443086, Russia, \\
2. Department of Mechanics and Mathematics, Saratov State University, \\
Astrakhanskaya 83, Saratov 410012, Russia, \\
e-mail: {\it BondarenkoNP@info.sgu.ru}

\end{document}